\pdfoutput=1
\RequirePackage{ifpdf}
\ifpdf % We are running pdfTeX in pdf mode
\documentclass[pdftex]{sigma}
\else
\documentclass{sigma}
\fi

\newcommand{\mb}[1]{\mathbb{#1}}
\newcommand{\mf}[1]{\mathfrak{#1}}

\newcommand{\tx}{\text}

\newcommand{\bmtx}[1]{\begin{bmatrix} #1 \end{bmatrix}}

\newcommand{\tensor}{\otimes}

\numberwithin{equation}{section}

\newtheorem{Theorem}{Theorem}[section]
\newtheorem{Corollary}[Theorem]{Corollary}
\newtheorem{Lemma}[Theorem]{Lemma}
\newtheorem{Proposition}[Theorem]{Proposition}
{ \theoremstyle{definition}
\newtheorem{Definition}[Theorem]{Definition}
\newtheorem{Remark}[Theorem]{Remark} }

\begin{document}

\allowdisplaybreaks

\newcommand{\arXivNumber}{1507.08365}

\renewcommand{\PaperNumber}{055}

\FirstPageHeading

\ShortArticleName{A Family of Finite-Dimensional Representations of GDAHA of Higher Rank}

\ArticleName{A Family of Finite-Dimensional Representations\\ of Generalized Double Af\/f\/ine Hecke Algebras\\ of Higher Rank}

\Author{Yuchen FU and Seth SHELLEY-ABRAHAMSON}
\AuthorNameForHeading{Y.~Fu and S.~Shelley-Abrahamson}
\Address{Department of Mathematics, Massachusetts Institute of Technology, \\
182 Memorial Drive, Cambridge, MA 02139, USA}
\Email{\href{mailto:yfu@mit.edu}{yfu@mit.edu}, \href{mailto:sethsa@mit.edu}{sethsa@mit.edu}}

\ArticleDates{Received April 20, 2016, in f\/inal form June 11, 2016; Published online June 14, 2016}

\Abstract{We give explicit constructions of some f\/inite-dimensional representations of ge\-ne\-ralized double af\/f\/ine Hecke algebras (GDAHA) of higher rank using $R$-matrices for $U_q(\mf{sl}_N)$. Our construction is motivated by an analogous construction of Silvia Montarani in the rational case. Using the Drinfeld--Kohno theorem for Knizhnik--Zamolodchikov dif\/ferential equations, we prove that the explicit representations we produce correspond to Montarani's representations under a monodromy functor introduced by Etingof, Gan, and Oblomkov.}


\Classification{20C08}

\section{Introduction}

Generalized double af\/f\/ine Hecke algebras of higher rank (GDAHA) are a family of algebras that generalize the well-known Cherednik algebras in representation theory and were f\/irst introduced by Etingof, Oblomkov, and Rains \ in \cite{Etingof2007749} (the rank~1 case) and~\cite{EGO} (the rank $n$ case). They have been related to del Pezzo surfaces in algebraic geometry and Calogero--Moser integrable systems. A degenerate version of GDAHA, known as the rational GDAHA, was also introduced in~\cite{EGO}. In~\cite{Montarani10}, Montarani introduces two constructions of f\/inite-dimensional representations of rational GDAHA, one using $D$-modules and the other using explicit Lie theoretic methods. The latter of these methods, involving an isotypic subspace of a tensor product with a tensor power of the vector representation of $\mf{sl}_N$, is similar in spirit to the Arakawa--Suzuki functor from the BGG category $\mathcal{O}$ of $\mf{sl}_N$-modules to the category of f\/inite-dimensional representations of degenerate af\/f\/ine Hecke algebras of type $A$ constructed in \cite{arakawasuzuki1998} and to the construction of representations of af\/f\/ine braid groups given by Orellana and Ram in \cite{orellanaram}. There has been much interest in and ef\/fort devoted to developing similar techniques for the representation theory of related algebraic objects, for instance by Calaque, Enriquez, and Etingof \cite{cce} for degenerate double af\/f\/ine Hecke algebras and by Jordan \cite{jordan} in the nondegenerate case, among others.

In this paper, we generalize the Montarani's Lie theoretic construction to the non-degenerate GDAHA using $R$-matrices for the quantum groups $U_q(\mf{sl}_N)$. Furthermore, we show that the explicit representations we produce are equivalent to the image of Montarani's representations under a monodromy functor constructed in \cite{EGO} which introduces an action of a nondegenerate GDAHA on a f\/inite-dimensional representation of a rational GDAHA. In the same sense that Jordan's work \cite{jordan} is a generalization of the work of Calaque, Enriquez, and Etingof \cite{cce} to the nondegenerate case, this paper generalizes Montarani's construction to the nondegenerate case.

The structure of this paper is as follows. In Section~\ref{background-section} we f\/ix the notation we use for quantum groups and review their basic properties and also introduce the def\/inition of GDAHA. Following that, our explicit construction of representations is given in Section~\ref{explicit-section}, along with its proof. The calculation that relates this representation with the monodromy representation given in \cite{EGO} is done in Section~\ref{monodromy-section}.
\section{Background}\label{background-section}
\subsection{Generalized double af\/f\/ine Hecke algebras of higher rank (GDAHA)}

We recall the def\/inition of GDAHA in \cite{EGO}. Let $D$ be a star-shaped f\/inite graph that is not f\/inite Dynkin with $m$ legs and leg lengths $d_1, \dots, d_m$ (number of vertices on each leg, including the center).

\begin{Definition}[GDAHA]\label{GDAHA-def}
	$H_{n}(D)$, the generalized double af\/f\/ine Hecke algebra of rank $n$ associa\-ted with graph $D$, is the associative algebra generated over $\mb{C}[u_{1, 1}^{\pm 1}, \dots, u_{1 d_1}^{\pm 1}, u_{2, 1}^{\pm 1}, \dots$, $u_{m, d_m}^{\pm 1}, t^{\pm 1}]$
by invertible generators $U_1, \dots, U_m, T_{1}, \dots, T_{n-1}$ and the following relations:
\begin{enumerate}\itemsep=0pt
\item $U_1 U_2 \cdots U_m T_1 T_2 \cdots T_{n-2} T_{n-1} T_{n-1} T_{n-2} \cdots T_2 T_1 = 1$;\label{GDAHA-rel1}
\item $T_i T_{i + 1} T_i = T_{i + 1} T_i T_{i + 1}$ for $1 \le i < n - 1$;\label{GDAHA-rel2}
\item $[T_i, T_j] = 0$ for $|i - j| > 1$;\label{GDAHA-rel3}
\item $[U_i, T_j] = 0$ for $1 < i \le n - 1, 1 < j \le m$;\label{GDAHA-rel4}
\item $[U_i, T_1 U_i T_1] = 0$ for $1 \le i \le m$;\label{GDAHA-rel5}
\item $[U_i, T_1^{-1} U_j T_1] = 0$ for $1 \le i < j \le m$;\label{GDAHA-rel6}
\item $\prod\limits_{j = 1}^{d_k} (U_k - u_{k, j}) = 0$ for $1 \le k \le m$;\label{GDAHA-rel7}
\item $T_i - T_i^{-1} = t - t^{-1}$ for $1 \le i \le n - 1$.\label{GDAHA-rel8}
\end{enumerate}
\end{Definition}

The rank-$n$ GDAHA $H_{n, m}$ can be seen as a quotient of the group algebra of the fundamental group of the conf\/iguration space of $n$ unordered points on the $m$-punctured sphere, where the quotient is by the eigenvalue relations \eqref{GDAHA-rel7} and \eqref{GDAHA-rel8}. From this perspective, $U_i$ is represented by a path in the conf\/iguration space in which one of the points loops around a missing point $\alpha_i$, and $T_i$ is represented by a typical braid group generator exchanging the positions of two points in the conf\/iguration. Note that the generators $T_i$ satisfy the relations of the f\/inite-type Hecke algebra of type $A_{n - 1}$.

\subsection[The quantum group $U_q(\mf{sl}_N)$]{The quantum group $\boldsymbol{U_q(\mf{sl}_N)}$}

Due to the large volume of literature and varying conventions regarding quantum groups, it is necessary to f\/ix the notations that we work with. We use conventions compatible with, for example, \cite{etingof1998lectures,jantzen1996lectures, KS}. Throughout this paper, it is assumed that $q$ is a nonzero complex number that is not a root of unity.

For every $n \in \mb{Z}$, let $[n]_q$ denote the associated symmetrized $q$-number
\begin{gather*}
[n]_q := \dfrac{q^n - q^{-n}}{q - q^{-1}}.
\end{gather*}
In particular, we have
\begin{gather*}
[n]_q = q^{n - 1} + q^{n - 3} + \cdots + q^{-n + 3} + q^{-n + 1}
\end{gather*}
for $n > 0$ and $[n]_q = -[-n]_q$ for all $n$. Def\/ine $q$-binomial coef\/f\/icients by
\begin{gather*}\bmtx{a \\ n}_q := \dfrac{[a]_q [a - 1]_q \cdots [a - n + 1]_q}{[1]_q [2]_q \cdots [n]_q}\end{gather*}
for all $a \in \mb{Z}$, $n \in \mb{N}$, so in particular $\bmtx{a \\ 0}_q = 1$ by the usual convention on empty products. Def\/ine $q$-factorials $[a]_q^{!}$ for $a \in \mb{N}$ by
\begin{gather*}[a]_q^! = [1]_q [2]_q \cdots [a]_q.\end{gather*} Note that $[0]_q^! = 1$.

\begin{Definition}[quantum group $U_q(\mf{sl}_N)$]
	Let $A = (a_{ij})$ be the Cartan matrix for the Lie algebra $\mf{sl}_N$. The quantum group $U_q(\mf{sl}_N)$ is the associative $\mb{C}$-algebra with generators $E_i$, $F_i$, $K_i$, $K_i^{-1}$ ($1 \le i < N$) and relations:
\begin{gather*}
[K_i, K_j] = 0, \\
K_i K_i^{-1} = 1 = K_i^{-1} K_i, \\
K_i E_j = q^{a_{ij}} E_j K_i, \\
K_i F_j = q^{-a_{ij}} F_j K_i, \\
[E_i, F_j] = \delta_{ij} \frac{K_i - K_i^{-1}}{q_i - q_i^{-1}}, \\
\sum_{r = 0}^{1 - a_{ij}} (-1)^r \bmtx{1-a_{ij} \\ r}_{q} E_i^{1 - a_{ij} - r} E_j E_i^r = 0,\quad \text{and} \\
\sum_{r = 0}^{1 - a_{ij}} (-1)^r \bmtx{1-a_{ij} \\ r}_{q} F_i^{1 - a_{ij} - r} F_j F_i^r = 0 \quad \text{if} \ \ i \neq j,
\end{gather*}
where $\delta_{ij}$ is the Kronecker delta and $[a, b] = ab - ba$.
\end{Definition}

We will use the Hopf algebra structure on $U_q(\mf{sl}_N)$ specif\/ied in the following standard proposition:

\begin{Proposition}[Hopf algebra structure on $U_q(\mf{sl}_N)$] Denote $U = U_q(\mf{sl}_N)$. The assignments below extend to unique algebra homomorphisms $(\Delta\colon  U \to U \tensor U,$ $\varepsilon\colon  U \to \mb{C},$ $\eta\colon  \mb{C} \to U,$ $S\colon U \to U^{\rm op})$ that give $U$ a Hopf algebra structure:
\begin{gather*}
\Delta(K_i) = K_i \tensor K_i,\qquad
\Delta(E_i) = E_i \tensor K_i + 1 \tensor E_i, \qquad \Delta(F_i) = F_i \tensor 1 + K_i^{-1} \tensor F_i, \\
\varepsilon(E_i) = \varepsilon(F_i) = 0, \qquad \varepsilon(K_i) = 1,\qquad
S(K_i) = K_i^{-1},\qquad  S(E_i) = - E_i K_i^{-1}, \\ S(F_i) = - K_i F_i.
\end{gather*}
\end{Proposition}

\subsection[$R$-matrices]{$\boldsymbol{R}$-matrices}

In this section we will f\/ix our conventions for $R$-matrices for $U_q(\mf{sl}_N)$. All statements below are well-known and their proofs can be found in \cite{KS} or \cite{kassel2012quantum}.

\begin{Definition}For a given pair $(V, V')$ of f\/inite-dimensional $U_q(\mf{sl}_N)$-modules, def\/ine the linear map $f\colon V \tensor V' \to V \tensor V'$ such that $f(v \tensor w) = q^{( \lambda, \mu )} (v \tensor w)$ if $v$ and $w$ have weights~$\lambda$ and~$\mu$, respectively, where $(\cdot, \cdot)$ is the standard pairing on the weight lattice.
\end{Definition}

\begin{Proposition}\label{r-matrix prop}
There exists an element
\begin{gather*}\widetilde{\mf{R}} \in 1 + U_q(\mf{sl}_n)_{> 0} \,\hat{\otimes}\, U_q(\mf{sl_n})_{< 0}\end{gather*}
in an appropriate completion $U_q(\mf{sl}_N) \hat{\otimes} U_q(\mf{sl}_N)$ such that the operators \begin{gather*}\mf{R} := f \circ \widetilde{\mf{R}},\qquad  R := P \circ \mf{R} = P \circ f \circ \widetilde{\mf{R}}\end{gather*} on tensor products of finite-dimensional $U_q(\mf{sl}_N)$-modules satisfy, for any finite-dimensional mo\-du\-les $V_1$, $V_2$, $V_3$:
	\begin{itemize}\itemsep=0pt
	\item $R\colon V_1 \otimes V_2 \rightarrow V_2 \otimes V_1$ is an isomorphism of $U_q(\mf{sl}_n)$-modules,
	\item $\mf{R}_{12,3} = \mf{R}_{13} \mf{R}_{23}$,
	\item $\mf{R}_{1,23} = \mf{R}_{13} \mf{R}_{12}$,
	\end{itemize}
where $P$ denotes the operator exchanging tensor factors and the subscripts indicate the tensor factor positions on which $\mf{R}$ acts.
 \end{Proposition}

For an explicit description of the element $\widetilde{\mf{R}}$, see \cite[Section 8.3.3]{KS}.

The operator $\mf{R}$ satisf\/ies the quantum Yang--Baxter equation:

\begin{Proposition}[QYBE]
$\mf{R}_{12} \mf{R}_{13} \mf{R}_{23} = \mf{R}_{23} \mf{R}_{13} \mf{R}_{12}$.
\end{Proposition}

\begin{Corollary}\label{R-form-braid-group}
For $1 \le i \neq j < n$, let $R_i = P_{i, i + 1} \circ \mf{R}_{i, i + 1}$, then we have $R_i R_{i + 1} R_i = R_{i + 1} R_i R_{i + 1}$ and $R_i R_j = R_j R_i$ for $|i - j| > 1$, and in particular the $R_i$ operators yield a~representation of the braid group $B_n$ on $V^{\tensor n}$ for any finite-dimensional $U_q(\mf{sl}_N)$-representation~$V$, and this representation is functorial in~$V$.
\end{Corollary}

\begin{Proposition}\label{R-form-Hecke-algebra}
Let $\mb{C}^N$ be the $N$-dimensional vector representation of $U_q(\mf{sl}_N)$. Then $R_i$ $(1 \le i < n)$ act on the space $\big(\mb{C}^N\big)^{\tensor n}$. The operators $q^{1/N}R_i$ on $\big(\mb{C}^N\big)^{\otimes n}$ act with eigenvalues~$q$,~$-q^{-1}$ and in particular define a representation of the Hecke algebra of type $A_{n - 1}$ with parameter~$q$.
\end{Proposition}

\begin{proof} This is Proposition 23 in Section 8.4.3 of \cite{KS}.\end{proof}

\subsection[Ribbon category structure on $U_q(\mf{sl}_N) \text{-}\mathrm{mod}_{\rm f.d.}$]{Ribbon category structure on $\boldsymbol{U_q(\mf{sl}_N) \text{-}\mathrm{mod}_{\rm f.d.}}$}\label{ribbon-section}

Recall that a type-1 representation of $U_q(\mf{sl}_N)$ is a representation $V$ such that $V$ has a weight decomposition $V = \oplus_\lambda V_\lambda$ where the direct sum is over the weight lattice for $\mf{sl}_N$ and $V_\lambda$ is the subspace
\begin{gather*}V_\lambda := \big\{v \in C\colon K_\mu(v) = q^{(\lambda, \mu)}v~\text{for all roots}~\mu\big\}.\end{gather*}
Let $U_q(\mf{sl}_N)\text{-}\mathrm{mod}_{\rm f.d.}$ denote the category of f\/inite-dimensional type-1 representations of~$U_q(\mf{sl}_N)$. We will only consider representations in this category.

The operator $R$, along with the Hopf algebra structure on $U_q(\mf{sl}_N)$, gives the category $U_q(\mf{sl}_N)\text{-}\mathrm{mod}_{\rm f.d.}$ of f\/inite-dimensional type-1 representations of $U_q(\mf{sl}_N)$ the structure of a braided rigid tensor category. In fact, $U_q(\mf{sl}_N)\text{-}\mathrm{mod}_{\rm f.d.}$ has even richer structure~-- it is a~\emph{ribbon category}. In particular, there exists an automorphism $\theta$ of the identity functor on $U_q(\mf{sl}_N)\text{-}\mathrm{mod}_{\rm f.d.}$ which is given by multiplication by $q^{(\lambda, \lambda + 2 \rho)}$ (where, as usual, $\rho$ is the half-sum of the positive roots) on any irreducible representation with highest weight $\lambda$, and $\theta$ satisf\/ies the compatibility
\begin{equation}
\theta_{V(\lambda) \otimes V(\mu)} = R^2 \circ \big(\theta_{V(\lambda)} \otimes \theta_{V(\mu)}\big),\label{eq:ribbon-equation}
\end{equation}
where $V(\lambda)$ and $V(\mu)$ are the irreducible highest weight representations with highest weights~$\lambda$ and~$\mu$, respectively, and $R$ is the $R$-matrix introduced above. An $h$-adic version of these statements can be found in Proposition~21 in Section~8.4.3 of~\cite{KS}, and it is routine to translate the result into the setting of numeric~$q$.

\section[Representations of GDAHA via $R$-matrices]{Representations of GDAHA via $\boldsymbol{R}$-matrices}\label{explicit-section}

\subsection{The construction}\label{main-construction}

Again, we assume $q$ is a nonzero complex number which is not a root of unity, and all representations of $U_q(\mf{sl}_N)$ we consider will be type-1 and f\/inite-dimensional.

Let $n \geq 1$ be an integer, let $V_1, \dots, V_m$ be irreducible f\/inite-dimensional highest weight modules for $U_q(\mf{sl}_N)$, let $V_{m + 1} = V_{m + 2} = \dots = V_{m + n} = \mb{C}^N$ be copies of its vector representation, and let
\begin{gather*}V = V_1 \tensor \cdots \tensor V_m \tensor V_{m + 1} \tensor \cdots \tensor V_{m + n}.\end{gather*}
Let $E$ be the $0$-isotypic component of $V$. For $1 \leq i < m + n$ let $R_i$ denote $\operatorname{id}^{\otimes i - 1} \otimes R \otimes \operatorname{id}^{\otimes m + n - i - 1}$.

\begin{Theorem}[main theorem]\label{main-theorem}
	Let $\lambda_1, \dots, \lambda_m \in \mb{C}$ be any complex numbers and let $c = (n + \sum_k \lambda_k)/N$. Then for $t = q$, some specific values of $u_{kj}$, and some graph $D$, the following formulas for $1 \leq i < n$ and $1 \leq k \leq m$ define a representation of the GDAHA of rank $n$ attached to $D$ on $E$:
	\begin{itemize}\itemsep=0pt
	\item $T_i = q^{1/N} R_{m + i}$, and
	\item $U_k = q^{2[(N - c) / m + \lambda_k / N]} R_{m} R_{m - 1} \cdots R_{k + 1} R_{k}^2 R_{k + 1}^{-1} \cdots R_{m - 1}^{-1} R_m^{-1}$.
	\end{itemize}
\end{Theorem}

Note that the $T_i$ are endomorphisms of $E$ because $V_{m + 1}, \dots, V_{m + n}$ are all copies of $\mb{C}^N$ and $R_{m + i}$ acts as a $U_q(\mf{sl}_N)$-module isomorphism. The $U_k$ act on $E$ for similar reasons. The dependence of the graph $D$ and the parameters $u_{kj}$ on the representations $V_1, \dots, V_m$ and the scalars $\lambda_1, \dots, \lambda_m \in \mb{C}$ is given later in equation~\eqref{eq:ukj-equation}.
\subsection{Validity of the def\/ining relations}

We need to verify that the relations \eqref{GDAHA-rel1}--\eqref{GDAHA-rel8} def\/ining GDAHA hold for the operators def\/ined in the theorem above. The relations~\eqref{GDAHA-rel2}, \eqref{GDAHA-rel3} and \eqref{GDAHA-rel8} among $T_i$ are the def\/ining relations of the type $A_{n - 1}$ Hecke algebra and hold by Proposition~\ref{R-form-Hecke-algebra}. Relation \eqref{GDAHA-rel4} holds trivially because the operators act on dif\/ferent tensor factors. Relations \eqref{GDAHA-rel5} and \eqref{GDAHA-rel6} follow from routine calculations using that the $R$-matrices $R_i$ satisfy the braid relations, and so we need only check relations~\eqref{GDAHA-rel1} and~\eqref{GDAHA-rel7}. We begin with relation~\eqref{GDAHA-rel7}, which restricts the eigenvalues of the operators $U_i$.

\subsubsection[Eigenvalues of $U_k$]{Eigenvalues of $\boldsymbol{U_k}$}\label{eigenvalue-subsubsection}

Recall from Section \ref{background-section} the functorial operator $\theta$ def\/ined on any $U_q(\mf{sl}_N)$-module and acting on any copy of the highest weight module $V(\lambda)$ by $q^{(\lambda, \lambda + 2 \rho)}$. From the compatibility between the~$R$ matrix and $\theta$ given in~\eqref{eq:ribbon-equation} that~$R^2$ acts as a scalar on each highest weight submodule of the tensor product of two highest weight modules. More precisely,

\begin{Lemma}	\label{R2-action-on-two}
Let $V(\mu)$ and $V(\mu')$ be highest weight representations of $U$ and let $V(\lambda) \subset V(\mu) \otimes V(\mu')$ be a copy of the simple highest weight module with highest weight~$\lambda$. Then~$R^2$ acts on~$V(\lambda)$ by the scalar $q^{- (\mu, \mu + 2 \rho) - (\mu', \mu' + 2 \rho) + (\lambda, \lambda + 2 \rho)}$.
\end{Lemma}

\begin{proof}Immediate from $\theta_{V(\mu) \tensor V(\mu')} = R^2 (\theta_{V(\mu)} \tensor \theta_{V(\mu')})$.
\end{proof}

For $1 \leq k \leq m$, let $\tilde{U}_k := R_m \cdots R_{k + 1}R_k^2R_{k + 1}^{-1} \cdots R_m^{-1}.$ Relation \eqref{GDAHA-rel7} requires that the eigen\-va\-lues of $U_k$ counted with multiplicity are among $u_{k, 1}, \dots, u_{k, d_k}$, so it is equivalent to show that~$\tilde{U}_k$ has all its eigenvalues counted with multiplicity appearing in \begin{gather*}q^{-2[(N - c) / m + \lambda_k / N]} u_{k, 1}, \quad \dots, \quad q^{-2[(N - c) / m + \lambda_k / N]} u_{k, d_k}.\end{gather*}
As $\tilde{U_k}$ is conjugate to $R_k^2$ via $R_m\cdots R_{k + 1}$, we need only compute the eigenvalues of $R_k^2$ acting on $V_k \otimes \mb{C}^N$.

Suppose $V_k$ is of highest weight $\mu_k$. As $\mb{C}^N$ is of highest weight $\epsilon_1 := (1, 0, \dots, 0)$, it follows from Lemma \ref{R2-action-on-two} that the eigenvalues of $R_k^2$ on $V_k \tensor \mb{C}^N$ are $q^{-(\mu_k, \mu_k + 2 \rho) - (\epsilon_1, \epsilon_1 + 2 \rho) + (\eta, \eta + 2 \rho)}$, where~$\eta$ ranges over the highest weights of the irreducible constituents appearing in the direct sum decomposition of $V_k \tensor \mb{C}^N$. Let $d_k$ be the number of non-isomorphic irreducibles appearing in $V_k \tensor \mb{C}^N$, and suppose their distinct highest weights are $\eta_{k, 1}, \dots, \eta_{k, d_k}$. Then, setting
\begin{gather}\label{eq:ukj-equation}
u_{k, j} = q^{2[(N - c) / m + \lambda_k / N] -(\mu_k, \mu_k + 2\rho) + (\eta_{k, j}, \eta_{k, j} + 2 \rho) - N + 1/N}
\end{gather}
provides the choice of $d_1, \dots, d_m$ and $u_{kj}$ so that relation~\eqref{GDAHA-rel7} holds (note that $(\epsilon_1, \epsilon_1 + 2 \rho) = N - 1/N$). In particular, the graph $D$ is one with $m$ legs and leg lengths $d_1, \dots, d_m$.

\subsubsection{The f\/irst relation}

If we expand relation \eqref{GDAHA-rel1} of the GDAHA, i.e., $U_1\cdots U_m T_1 \cdots T_{n - 1}T_{n - 1} \cdots T_1 = 1$, in terms of the operators given in Theorem~\ref{main-theorem}, we see that it simplif\/ies to the following equality:
\begin{gather*}q^{2N - 2/N} R_m R_{m - 1} \cdots R_1 (R_1 R_2 \cdots R_m R_{m + 1} \cdots R_{m + n - 1}) R_{m + n - 1} R_{m + n - 2} \cdots R_{m + 1} = \operatorname{id}_E.\end{gather*}
Now, suppose we have the following identity:
\begin{gather*}R_1 R_2 \cdots R_{m + n - 1} R_{m + n - 1} R_{m + n - 2} \cdots R_2 R_1 = q^{-2N + 2/N} \operatorname{id}_{E'},\end{gather*}
where $E'$ is the 0-isotypic subspace of $V' := V_{m + 1} \tensor V_1 \tensor \cdots \tensor V_{m} \tensor V_{m + 2} \tensor \cdots \tensor V_{m + n}$, then one can write \begin{gather*}R_1 R_2 \cdots R_{m + n - 1} = q^{-2N + 2 / N} R_1^{-1} R_2^{-1} \cdots R_{m + n - 1}^{-1}\end{gather*} as maps from the 0-isotypic component of $V_{1} \tensor \cdots \tensor V_{m} \tensor V_{m + 2} \tensor \cdots \tensor V_{m + n} \tensor V_{m + 1}$ to the 0-isotypic component $E'$ of $V'$. Substituting this into the parenthesized part in the expression above, we see it immediately proves relation \eqref{GDAHA-rel1}. Thus we have reduced relation~\eqref{GDAHA-rel1} to the following lemma:

\begin{Lemma}\label{J1-acts-by-scalar}
The operator $R_1 R_2 \cdots R_{m + n - 1} R_{m + n - 1} R_{m + n - 2} \cdots R_2 R_1\colon V' \to V'$ acts on the zero isotypic component of $V'$ by $q^{-2N + 2/N}$.
\end{Lemma}

\begin{proof} It follows from properties \eqref{GDAHA-rel2} and \eqref{GDAHA-rel3} of $R$-matrices in Proposition \ref{r-matrix prop} that we have
\begin{gather*}R_1 \cdots R_{m + n - 1} = R_{12\cdots(m + n - 1), m + n}\end{gather*}
and
\begin{gather*}R_{m + n - 1} \cdots R_1 = R_{1, 23\cdots (m + n)},\end{gather*}
where the notations $R_{12\cdots(m + n - 1), m + n}$ and $R_1 = R_{1, 23\cdots (m + n)}$ are as in Proposition~\ref{r-matrix prop}, i.e., these represent the $R$-matrices with f\/lip associated to the bracketings $(\bullet^{m + n - 1})\bullet$ and $\bullet(\bullet^{m + n - 1})$ respectively. Let
\begin{gather*}V_1 \otimes \cdots \otimes V_m \otimes V_{m + 2} \otimes \cdots \otimes V_{m + n} = \bigoplus_i W_i\end{gather*}
be a decomposition of $V_1 \otimes \cdots \otimes V_m \otimes V_{m + 2} \otimes \cdots \otimes V_{m + n}$ into irreducibles and similarly let
\begin{gather*}\mb{C}^N \otimes W_i = \bigoplus_{j} Z_{ij}\end{gather*}be a decomposition of $\mb{C}^N \otimes W_i$ into irreducibles. If $W_i$ is of highest weight $\mu_i$ and if $Z_{ij}$ is of \mbox{highest} weight $\nu_{ij}$, by Lemma~\ref{R2-action-on-two} and the f\/irst sentence of this proof, $R_1 {\cdots} R_{m + n - 1} R_{m + n - 1} {\cdots} R_1$ acts on $Z_{ij}$ by the scalar $q^{-(\epsilon_1, \epsilon_1 + 2\rho) - (\mu_i, \mu_i + 2\rho) + (\nu_{ij}, \nu_{ij} + 2\rho)}.$ However, we are only concerned with those $Z_{ij}$ with highest weight~0. By Pieri's rule, $V(0)$ can appear as a constituent of $\mb{C}^N \otimes V(\mu_i)$ only if $\mu_i$ is the weight $(1, \dots, 1, 0)$, i.e., when $V(\mu_i)$ is labeled by the Young diagram associated to the partition $1^{N - 1}$ \cite{fulton1991representation}. In this case the power of $q$ we just computed becomes $q^{-2N + 2/N}$, as needed. The lemma, and hence Theorem~\ref{main-theorem}, follow.\end{proof}

\section{Equivalence with the monodromy representation}\label{monodromy-section}

\subsection{The monodromy functor}

In \cite{EGO}, the authors introduced a certain connection of Knizhnik--Zamolodchikov type whose mo\-no\-dromy def\/ines a functor from the category of f\/inite-dimensional representations of a rational GDAHA (rGDAHA) $B_n$ to the category of f\/inite-dimensional representations of a corresponding non-degenerate GDAHA $H_n$ attached to the same diagram $D$ with m legs of lengths $d_1, \dots, d_m$. Recall that the rGDAHA $B_n$ is the algebra generated over $\mb{C}[\gamma_{1, 1}, \dots, \gamma_{m, d_m}, \nu]$ by elements $Y_{i, k}$ ($1 \le i \le n, \, 1 \le k \le m$) and the symmetric group $S_n$, such that the following hold for every $i, j, h \in [1, n]$ with $i \neq j$ and every $k, l \in [1, m]$:
\begin{enumerate}\itemsep=0pt
\item $s_{ij} Y_{i, k} = Y_{j, k} s_{ij}$;
\item $s_{ij} Y_{h, k} = Y_{h, k} s_{ij}$ if $h \neq i, j$;
\item $\prod\limits_{j = 1}^{d_k} (Y_{i, k} - \gamma_{k,j}) = 0$;
\item $\sum\limits_{j = 1}^{m} Y_{i, j} = \nu \sum\limits_{j \neq i} s_{ij}$;
\item $[Y_{i, k}, Y_{j, k}] = \nu (Y_{i, k} - Y_{j, k}) s_{ij}$;
\item $[Y_{i, k}, Y_{j, l}] = 0$ if $k \neq l$.
\end{enumerate}

Let us recall the monodromy functor from \cite{EGO}. Fix a f\/inite-dimensional $B_n$ module $M$. For convenience, let $\alpha_1, \dots, \alpha_m \in \mb{C}$ be distinct points def\/ined by $\alpha_i = -m - 1 + i$ and choose the basepoint $\textbf{z}_0 = (1, \dots, n)$ in the ordered conf\/iguration space $\tx{Conf}_n(\mb{C}\backslash\{\alpha_1, \dots, \alpha_m\})$ of $n$-points in the $m$-punctured plane. In \cite{EGO} the Knizhnik--Zamolodchikov-style connection
\begin{gather}\label{KZ-rGDAHA}
\nabla_{\rm EGO} := d - \sum_{i = 1}^n \left(\sum_{k = 1}^m \frac{Y_{i, k}}{z_i - \alpha_k} - \sum_{j \neq i} \frac{\nu s_{ij}}{z_i - z_j}\right) dz_i
\end{gather}
is introduced on the trivial vector bundle $E_M$ with f\/iber $M$ over $\tx{Conf}_n(\mb{C}\backslash\{\alpha_1, \dots, \alpha_m\})$. It follows readily from the def\/ining relations of $B_n$ and a calculation that the connection $\nabla$ is f\/lat and has trivial residue at $\infty$. This connection is visibly $S_n$-equivariant and so descends to a connection on the unordered conf\/iguration space $\tx{UConf}_n(\mb{C}\backslash\{\alpha_1, \dots, \alpha_m\})$. The residue condition implies that the monodromy of this connection at $\textbf{z}_0$ def\/ines a representation
\begin{gather*}\rho_M\colon \  \pi_1\big(\tx{UConf}_n\big(\mb{CP}^1\backslash\{\alpha_1, \dots, \alpha_m\}\big), \textbf{z}_0\big) \to \operatorname{Aut}(M).\end{gather*}

The algebra $H_n$ may be interpreted as the quotient of the group algebra\begin{gather*}\mb{C}\pi_1\big(\tx{UConf}_n\big(\mb{CP}^1\backslash\{\alpha_1, \dots, \alpha_m\}\big), \textbf{z}_0\big)\end{gather*}by the eigenvalue relations \eqref{GDAHA-rel7} and \eqref{GDAHA-rel8} of Def\/inition \ref{GDAHA-def}, where $U_i$ and $T_i$ are the generators represented by the following loops:
\begin{center}
	\includegraphics[scale=0.3]{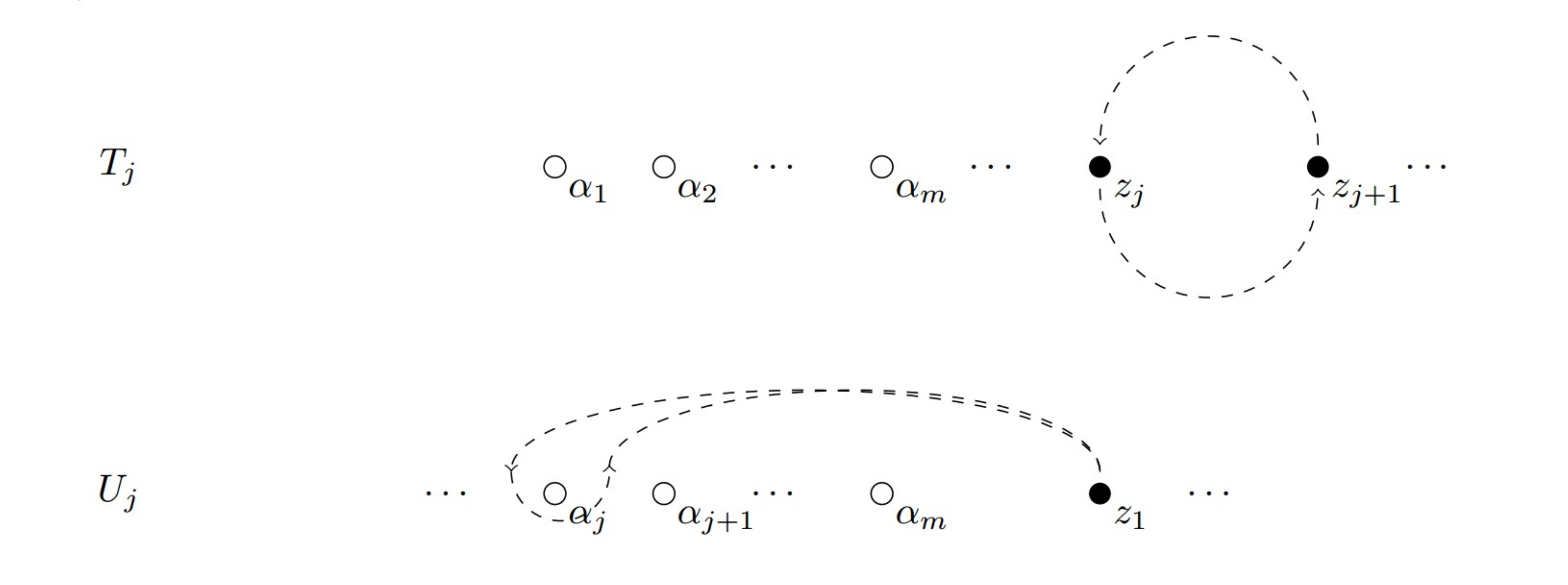}
\end{center}

In \cite[Section 4.2]{EGO} it is shown that the monodromy operators $\rho_M(T_i)$ and $\rho_M(U_k)$ satisfy the eigenvalue conditions \eqref{GDAHA-rel7} and \eqref{GDAHA-rel8} for $t = e^{- \pi i \nu}$, $u_{k,j} = e^{2 \pi i \gamma_{k, j}}$ and therefore def\/ine a f\/inite-dimensional representation of $H_n$ with these parameter values. This construction is clearly functorial in~$M$ and def\/ines a functor
\begin{gather*}
F \colon \ B_n\text{-}\mathrm{mod}_{\rm f.d.} \rightarrow H_n\text{-}\mathrm{mod}_{\rm f.d.},\end{gather*}
where as before $\text{-}\mathrm{mod}_{\rm f.d.}$ denotes the category of f\/inite-dimensional representations. Note that this functor is the identity at the level of vector spaces.

\subsection{Montarani's rGDAHA representation}
We now recall the construction of a family of f\/inite-dimensional representations of the rGDAHA $B_n$ given in \cite[Section~5]{Montarani10}. Let $V_1, \dots, V_m$ be irreducible f\/inite-dimensional representations of~$\mf{gl}_N$, and let $\lambda_1, \dots, \lambda_m \in \mb{C}$ be the scalars by which the identity matrix $I \in \mf{gl}_N$ acts, respectively. Let $\chi$ be a character of $\mf{gl}_N$, and denote the $\chi$-isotypic subspace of $V_1 \tensor \cdots \tensor V_m \tensor \big(\mb{C}^N\big)^{\tensor n}$ by
\begin{gather*}E_{n, \chi} := \big\{v \in V_1 \tensor \cdots \tensor V_m \tensor \big(\mb{C}^N\big)^{\tensor n} \,|\, xv = \chi'(x) v~\forall\, x \in \mf{gl}_N\big\}.\end{gather*}
Let $c \in \mb{C}$ be the scalar such that $\chi = c \operatorname{Tr}_{\mf{gl}_N}$. Note that if $E_{n, \chi} \neq 0$ then we have the relation $n + \sum\limits_{j = 1}^{m} \lambda_j = c N$ def\/ining $c$ in Theorem~\ref{main-theorem}.

Let $\Omega^{\mf{gl}_N} \in \mf{gl}_N \tensor \mf{gl}_N$ be the Casimir tensor of $\mf{gl}_N$, and let $\Omega^{\mf{gl}_N}_{ij}$ represent the tensor that acts as $\Omega^{\mf{gl}_N}$ on the $i$th and $j$th tensor factors and as identity on other factors.

\begin{Theorem}[\protect{\cite[Theorem 5.1]{Montarani10}}]\label{montarani-theorem}
	For any choice of $\nu \in \mb{C}$, the assignments $s_{ij} = \Omega^{\mf{gl}_N}_{m + i, m + j}$, and $Y_{i, k} = - \nu\big(\Omega^{\mf{gl}_N}_{k, m + i} + \frac{N - c}{m}\big)$ define a representation of the rGDAHA $B_n$ on $E_{n, \chi'}$ for an appropriate parameter value $\gamma$.
\end{Theorem}

\subsection{Equivalence of the representations}

We may apply the monodromy functor of~\cite{EGO} to the representation of the rGDAHA~$B_n$ in the previous theorem to produce a representation $F(E_{n, \chi})$ of a corresponding GDAHA~$H_n$. On the other hand, given the representations $V_1, \dots, V_{m + n}$ of~$\mf{gl}_N$, we can extract the constants $\lambda_1, \dots, \lambda_n$ as above. When $\nu \notin \mb{Q}$, the associated parameter $t = e^{-\pi i \nu}$ is a nonzero complex number which is not a root of unity. Let $q = t$. If $V_i^q$ denotes the irreducible representation of~$U_q(\mf{sl}_N)$ with highest weight corresponding to the highest weight of~$V_i$, then from the $V_i^q$ and the $\lambda_i$ we may produce, using Theorem \ref{main-theorem}, a f\/inite-dimensional representation $E$ of an associated non-degenerate GDAHA $H_n'$.

\begin{Theorem}\label{comparison-theorem} For $\nu \notin \mb{Q}$, the parameters of the GDAHAs~$H_n$ and $H_n'$ agree, and the representations $F(E_{n, \chi})$ and $E$ are equivalent.\end{Theorem}

\begin{proof} Let us f\/irst check that the eigenvalue parameters agree. The parameters $t$ for $H_n$ and~$H_n'$ agree and equal $e^{-\pi i \nu}$ by def\/inition. Let $\mu_k$ denote the highest weight of $V_k$, let $d_k$ be the number of non-isomorphic irreducible subrepresentations in $V_k \tensor \mb{C}^N = \bigoplus_{j} W_j$, and let their highest weights be $\eta_{k, 1}, \dots, \eta_{k, d_k}$. Then, by \cite[Lemma~5.2]{Montarani10}, the graph $D$ attached to the rGAHA~$B_n$ in Theorem~\ref{montarani-theorem} has leg lengths $d_1, \dots, d_m$ and parameter values $\gamma_{k, j} = -\nu(w_j + (N - c) / m)$ where $w_j$ is the eigenvalue of $\Omega^{\mf{gl}_N}$ on $W_j$. Therefore, with $q = e^{- i \pi \nu}$, the $u_{kj}$ parameter for the algebra $H_n$ is given by
\begin{gather*}u_{kj} := e^{2 \pi i \gamma_{k, j}} = q^{2(N - c) / m + (\eta_{k, j}, \eta_{k, j} + 2 \rho) - (\mu_k, \mu_k + 2 \rho) - (N - 1/N) + 2 \lambda_j / N},\end{gather*} where the term $2 \lambda_j / N$ comes from the discrepancy $\Omega^{\mf{gl}_N} = \Omega^{\mf{sl}_N} + \frac{1}{N} I \tensor I$ between the Casimir tensors for $\mf{gl}_N$ and $\mf{sl}_N$. This agrees with the parameter values for~$H_n'$ obtained in Section~\ref{eigenvalue-subsubsection}.

The strategy for proving the equivalence statement in Theorem~\ref{comparison-theorem} is to relate the connection $\nabla_{\rm EGO}$ to the classical KZ connection and to use the Drinfeld--Kohno theorem to relate the monodromy of the latter connection to $R$-matrices for $U_q(\mf{sl}_N)$. Let $V_1, \dots, V_{m + n}$ be f\/inite-dimensional irreducible $\mf{gl}_N$-representations as in Montarani's construction, with $V_{m + i}$ a copy of the vector representation $\mb{C}^N$ for $1 \leq i \leq n$. We have the associated KZ-connection
\begin{gather*}
\nabla_{\rm KZ} := d + \nu \sum_{1 \leq j \leq m + n,\ j \neq i} \frac{\Omega_{ij}^{\mf{sl}_N}}{z_i - z_j}dz_i
\end{gather*}
on the trivial vector bundle $E_{Y_{m + n}}$ with f\/iber $E_{n, \chi}$ over the unordered conf\/iguration space \begin{gather*}Y_{m + n} := \big\{(z_1, \dots, z_{m + n}) \in \mb{C}^{m + n} \colon z_i \neq z_j \text{ for all } i \neq j\big\}\end{gather*} of $m + n$ points in $\mb{C}$. Here $\Omega^{\mf{sl}_N}_{ij}$ denotes the Casimir tensor for $\mf{sl}_N$ acting on the~$i^{\rm th}$ and~$j^{\rm th}$ tensor factors and we view the $\mf{gl}_N$ representations as $\mf{sl}_N$ representations. It is well-known, and easy to check, that this connection is f\/lat. Observe that $\nabla_{\rm KZ}$ is $S_n$-equivariant, where $S_n$ acts on $Y_{m + n}$ by permuting the last $n$ coordinates and on the f\/iber by permuting the last $n$ tensor factors.

Let $Y := \tx{Conf}_n(\mb{C}\backslash\{\alpha_1, \dots, \alpha_m\})$ be the space on which $\nabla_{\rm EGO}$ is def\/ined. There is a natural $S_n$-equivariant map\begin{gather*}r \colon \  Y \rightarrow Y_{m + n}\end{gather*} given by $(z_1, \dots, z_n) \mapsto (\alpha_1, \dots, \alpha_m, z_1, \dots, z_n)$. Pulling back the connection $\nabla_{\rm KZ}$ along $r$ we obtain the $S_n$-equivariant f\/lat connection
\begin{gather*}
r^*\nabla_{\rm KZ} = d + \nu\sum_{i = 1}^{n} \left(\sum_{k = 1}^{m} \frac{\Omega_{k, m + i}^{\mf{sl}_N}}{z_i - \alpha_k} + \sum_{1 \le j \neq i \le n} \frac{\Omega^{\mf{sl}_N}_{m + i, m + j}}{z_i - z_j}\right) dz_i\end{gather*}
on the trivial vector bundle $E_Y$ with f\/iber $E_{n, \chi}$ over $Y$. On the other hand, inserting the operators $Y_{i, k}$ and $s_{i,j}$ on $E_{n, \chi}$ def\/ined in Theorem~\ref{montarani-theorem} into the connection $\nabla_{\rm EGO}$ def\/ined in equation \eqref{KZ-rGDAHA}, we obtain
\begin{gather*}
\nabla_{\rm EGO} = d + \nu\sum_{i = 1}^n\left(\sum_{k = 1}^{m} \frac{\Omega_{k, m + i}^{\mf{gl}_N} + \frac{N - c}{m}}{z_i - \alpha_k} + \sum_{1 \le j \neq i \le n} \frac{\Omega^{\mf{gl}_N}_{m + i, m + j}}{z_i - z_j}\right)dz_i.
\end{gather*}
Both connections above are f\/lat $S_n$-equivariant connections on the trivial vector bundle $E_Y$ over $Y$. To relate their monodromy, f\/irst recall that $\Omega^{\mf{sl}_N}$ and $\Omega^{\mf{gl}_N}$ are related by the equation $\Omega^{\mf{gl}_N} = \Omega^{\mf{sl}_N} + \frac{1}{N}I \otimes I$ where $I \in \mf{gl}_N$ is the identity matrix. As $I$ acts on $V_i$ by $\lambda_i$ for $1 \leq i \leq m$ and as $1$ for $m + 1 \leq i \leq m + n$, we have
\begin{gather}\label{connection-relationship}
\nabla_{\rm EGO} = r^*\nabla_{\rm KZ} + \nu \sum_{i = 1}^n\left(\sum_{k = 1}^{m} \frac{\lambda_k/N + (N - c)/m}{z_i - \alpha_k} + \sum_{1 \le j \neq i \le n} \frac{1/N}{z_i - z_j}\right)dz_i.
\end{gather}

Note that the connection
\begin{gather}\label{diff-connection}
\nabla_{\rm dif\/f} := d + \nu \sum_{i = 1}^n\left(\sum_{k = 1}^{m} \frac{\lambda_k/N + (N - c)/m}{z_i - \alpha_k} + \sum_{1 \le j \neq i \le n} \frac{1/N}{z_i - z_j}\right)dz_i
\end{gather}
on $E_Y$ is itself f\/lat and scalar-valued, so it follows from equation~\eqref{connection-relationship} that the parallel transport operator associated to a path $\gamma$ in $Y$ for the connection $(E_Y, \nabla_{\rm EGO})$ is obtained by multiplying the parallel transport operator associated to $\gamma$ for the connection $(E_Y, r^*\nabla_{\rm KZ})$ by the scalar-valued parallel transport operator associated to $\gamma$ for the connection $(E_Y, \nabla_{\rm dif\/f})$. By inspection of the residues in $\nabla_{\rm dif\/f}$ in equation~\eqref{diff-connection}, it follows that for the loops $U_k$ and $T_i$ about $\textbf{z}_0$ in~$Y/S_n$, the monodromies $\mu_{\nabla_{\rm EGO}}$ and $\mu_{\nabla_{r^*{\rm KZ}}}$ for the connections $\nabla_{\rm EGO}$ and $r^*\nabla_{\rm KZ}$ are related by
\begin{gather}\label{uk-scale}
\mu_{\nabla_{\rm EGO}}(U_k) = q^{2(\lambda_k/m + (N - c)/m)}\mu_{\nabla_{r^*{\rm KZ}}}(U_k)
\end{gather} and
\begin{gather}\label{ti-scale}
\mu_{\nabla_{\rm EGO}}(T_i) = q^{1/N}\mu_{\nabla_{r^*{\rm KZ}}}(T_i),
\end{gather}
where as before $q$ is def\/ined by $q = e^{-\pi i \nu}$.

All that remains is to relate the monodromy operators $\mu_{\rm KZ}(U_k)$ and $\mu_{\rm KZ}(T_i)$ to the $R$-matrix expressions appearing in Theorem~\ref{main-theorem} using the Drinfeld--Kohno theorem. The original formulation of the Drinfeld--Kohno theorem as stated in~\cite{drinfeld1989quasi} was for the $\hbar$-adic quantum group $U_\hbar(\mf{g})$, but here we need a version of this theorem for $U_q(\mf{sl}_N)$ for $q$ a nonzero complex number which is not a root of unity (Drinfeld's work in \cite{drinfeld1989quasi} was a generalization of previous results obtained by Kohno~\cite{kohno1987} for $\mf{sl}_N$). Such a~result was obtained in~\cite{kazhdan1991affine} and an exposition can be found, for example, in \cite[Theorem 8.6.4]{etingof1998lectures}. Similarly to \cite[Corollary 8.6.5]{etingof1998lectures}, this theorem immediately implies in particular that for $\nu \notin \mb{Q}$, the monodromy representation $\mu_{\nabla_{\rm KZ}}$ of $\pi_1(Y_{m + n}/S_n)$ on $E_{n, \chi}$ at $\textbf{z}_0$ is equivalent to the representation given by $R$-matrix expressions in which the class of a loop~$\gamma$ at~$\textbf{z}_0$ acts by the product of $R$-matrices $R_{i_1}\cdots R_{i_l}$ whenever $[\pi_*\gamma]$ factors as $\sigma_{i_1}\cdots \sigma_{i_l}$ in the braid group $B_{n + m} := \pi_1(Y_{m + n}/S_{m + n})$. Here $\sigma_i$ is the standard~$i^{\rm th}$ generator of $B_i$ (counterclockwise half-loop around the hyperplane $z_i = z_{i + 1}$), $R_i$ is the $R$-matrix for $U_q(\mf{sl}_N)$ with $q = e^{-\pi i \nu}$, and $\pi\colon Y_{m + n}/S_n \rightarrow Y_{m + n}/S_{m + n}$ is the natural projection. As $\mu_{\nabla_{r^*{\rm KZ}}}(\gamma) = \mu_{\nabla_{\rm KZ}}(r_*\gamma)$ for any loop~$\gamma$, Theorem~\ref{comparison-theorem} now follows from equations~\eqref{uk-scale} and~\eqref{ti-scale} and the observation that $[\pi_*r_*T_i] = \sigma_{m + i}$ and $[\pi_*r_*U_k] = \sigma_m \cdots \sigma_k \sigma_k \sigma_{k + 1}^{-1} \cdots \sigma_m^{-1}$.
\end{proof}

\begin{Remark} A similar approach may be used for some $\nu \in \mb{Q}$ using Part~IV of~\cite{kazhdan1991affine}, but the analogous statements more complicated due to the failure of $U_q(\mf{sl}_N)\textrm{-mod}_{f.d}$ to be semisimple in this case. If the representations $V_1, \dots, V_m$ are f\/ixed, then there is an analogous construction via $R$-matrices for $\nu \in \mb{Q}$ as long as the denominator of $\nu$ is suf\/f\/iciently large.
\end{Remark}

\subsection*{Acknowledgements}

A portion of this work was conducted at the 2015 MIT Summer Program in Undergraduate Research and the MIT Undergraduate Research Opportunities Program, for which the second author was a graduate student mentor. We thank Pavel Etingof for many useful discussions and for suggesting the direction of this work. We also thank the anonymous referees for their valuable remarks and suggestions.

\pdfbookmark[1]{References}{ref}
\LastPageEnding

\end{document}